\documentclass[letterpaper]{amsart}
\usepackage[OT1]{fontenc}

\usepackage[foot]{amsaddr}

\usepackage{cite}
\usepackage{amsfonts}
\usepackage{amssymb}
\usepackage{tikz}
\usetikzlibrary{calc}
\usepackage{forest}
\usepackage{enumitem}
\usepackage{xcolor}
\usepackage{graphicx}
\usepackage[hidelinks]{hyperref}

\theoremstyle{plain}
\newtheorem{theorem}{Theorem}
\newtheorem{proposition}[theorem]{Proposition}
\newtheorem{lemma}[theorem]{Lemma}
\newtheorem{corollary}[theorem]{Corollary}

\theoremstyle{definition}

\newtheorem{example}[theorem]{Example}

\theoremstyle{remark}
\newtheorem*{remark}{Remark}

\newcommand{\dto}{\stackrel{d}{\longrightarrow}}
\newcommand{\floor}[1]{\lfloor #1 \rfloor}
\newcommand{\Di}{\,\mathrm{d}}
\newcommand{\IP}{\mathbb{P}}
\newcommand{\Poi}{\mathsf{Poi}}
\newcommand{\NB}{\mathsf{NB}}

\DeclareMathOperator{\PF}{PF}
\DeclareMathOperator{\PPF}{PPF}

\title[Cycles in Multiset Permutations and Parking Functions]{On Cycles in Multiset Permutations, Parking Functions, and Related Structures}

\author[C. Buchanan]{Calum {Buchanan}$^*$}
\address{$^*$Department of Mathematics, University of Denver, Denver, CO 80208, USA}
\email{calum.buchanan@du.edu}

\author[F. Burghart]{Fabian {Burghart}$^\dag$}
\address{$^\dag$Institute of Discrete Mathematics, TU Graz, Steyrergasse 30, 8010 Graz, Austria}
\email{fabian.burghart@tugraz.at}

\author[S. Wagner]{Stephan {Wagner}$^{\dag, \ddag}$}
\address{$^\ddag$Department of Mathematics, Uppsala University, Box 480, 751 06 Uppsala, Sweden}
\email{stephan.wagner@tugraz.at}

\author[M. Yin]{Mei {Yin}$^*$}
\email{mei.yin@du.edu}

\thanks{{\em Key words and phrases}: parking function, multiset permutation, cycle type, cyclic point, terminal closer, equivalence of ensembles}

\thanks{{\em Funding}: F.\;Burghart is supported by the Austrian Science Fund (FWF) [10.55776/F1002]. S.\;Wagner is supported by the Swedish research council (VR), grant 2022-04030. M.\;Yin is supported by the Simons Travel Support for Mathematicians Grant 00007227.}

\thanks{This manuscript has been accepted for publication in the proceedings of the 37th International Conference on Probabilistic, Combinatorial and Asymptotic Methods for the Analysis of Algorithms (AofA 2026).}

\begin{document}

\begin{abstract}
In this paper we study cycles in multiset permutations and parking functions. As combinatorial objects, multiset permutations are essential building blocks for mappings and permutations, while parking functions lie between mappings and permutations. We take both algebraic and analytic views in our investigation and present exact as well as asymptotic results. We point to a surprising correspondence between two statistics on multiset permutations, terminal closers and cyclic points, shedding light on the combinatorial structure.
\end{abstract}

\maketitle

\section{Introduction}
Let $[n]$ denote the set of positive integers $\{1, \dots, n\}$. In this paper, we study several related combinatorial objects:
\begin{itemize}
\item Let $S_n$ be the set of all {\em permutations} of $[n]$, so that $|S_n|=n!$.

\item A {\em mapping} of length $n$ is any function $f$ from $[n]$ to $[n]$. Let $\mathcal{F}_n$ denote the set of all mappings on $[n]$, so that $|\mathcal{F}_n|=n^n$. 

\item \emph{Parking functions} and \emph{prime parking functions} will be defined later in Sections \ref{sec:parking} and \ref{sec:prime_parking}. As we will see, they lie between mappings and permutations.

\item Let $M$ be an $n$-element multiset of integers taken from $[n]$. A {\em multiset permutation} of $M$ is 
a sequence $(w_1, \ldots, w_n)$ 
such that the multiplicity of each element $j$ in $M$ is $|\{i : w_i = j\}|$.
Let $S_M$ denote the set of all permutations of $M$.
For example, $(2,1,4,2) \in S_{\{1,2,2,4\}}$.
\end{itemize}

A common feature of all these sets of combinatorial objects is \emph{invariance under permutation}: if one element belongs to a set, then so does any of its permutations. Multiset permutations can be seen as building blocks of the others in that each of them is a disjoint union of $S_M$'s. 

The study of cycles in random mappings and permutations has a rich history. We refer the reader to Arratia, Barbour, and Tavar\'{e} \cite{Arratia2003} and Flajolet and Odlyzko \cite{Flajolet1990} for a comprehensive treatise. Multiset permutations and parking functions are less studied, and thus will be the central focus of our paper. Note that while our investigation into parking functions follows the classical interpretation of digraphs as done for mappings and permutations as in \cite{Paguyo2025} and \cite{Rubey2025}, there is an alternate interpretation of parking function digraphs, investigated in \cite{King2019} and \cite{Lackner2016}, by considering the vertices as the parking spaces and the directed edges as the one-way streets.

We will take both algebraic and analytic views in our investigation. Sections \ref{sec:multi}, \ref{sec:mp}, \ref{sec:parking}, and \ref{sec:prime_parking} respectively present exact results on multiset permutations, mappings and permutations, parking functions, and prime parking functions. Section \ref{sec:aa} on the other hand is dedicated to asymptotic analysis, demonstrating the asymptotic \emph{equivalence of ensembles} between parking functions and mappings concerning the number of cyclic points in their associated connected digraphs.

We point to a surprising correspondence between \emph{terminal closers} and \emph{cyclic points} in multiset permutations in Section \ref{subsec:tc}, which in turn leads to the correspondence between terminal closers and cyclic points in mappings, permutations, and parking functions. Since parking functions describe a dynamic parking process, the equality in distribution between terminal closers and cyclic points is even more intriguing than in the context of mappings and permutations. Having $k$ terminal closers in a parking function ensures that the first $k$ cars are lucky (a car is {\em lucky} if it parks at its preferred spot) and that the $(k+1)$st car, having the same preference as one of the first $k$ cars, is unlucky and causes a collision. Thus the expected number of lucky cars in a parking function is lower bounded by the expected number of cyclic points, and moreover, our asympotic results on cyclic points show that the number of lucky cars until the first unlucky car is on average of the order of magnitude $\sqrt{n}$.

\section{Multiset permutations}
\label{sec:multi}
In this section, we introduce relevant terminology and background. Since multiset permutations are essential building blocks for mappings, permutations, and parking functions, results presented in this section will serve as a basis for our derivations throughout the paper.

\subsection{Cyclic points}
Given a multiset permutation $w= (w_1, \dots, w_n) \in S_M$, let $G_w$ denote the {\em digraph} on the vertex set $[n]$ with edges $i \rightarrow w_i$. It is well known that the digraph $G_w$ associated to any $w$ can be obtained from a disjoint union of directed cycles on some subset $C$ of $[n]$ (the set of {\em cyclic points}) by selecting some vertices in $C$ and rooting to them some trees whose edges are all directed towards the root. 
For $k\geq 1$, a directed $k$-cycle in $G_w$ is a {\em cycle of length $k$ in $w$}, a subset $\{j_1, \ldots, j_k\}$ of $[n]$ such that $w_{j_i} = j_{i + 1}$ for each $i \in [k - 1]$ and $w_{j_k} = j_1$.

We use the notation $\lambda \vdash m$ to indicate that $\lambda$ is a partition of $m$. In other words, if $\lambda=(1^{k_1(\lambda)}, 2^{k_2(\lambda)}, \dots)$, where $k_i(\lambda)$ denotes the number of $i$'s in $\lambda$, then $\sum_i ik_i(\lambda) = m$. Given a partition $\lambda$, let $a(M,\lambda)$ be the number of $w\in S_M$ such that $G_w$ has cycle type 
$\lambda=(1^{k_1(\lambda)}, 2^{k_2(\lambda)}, \dots)$, i.e., $G_w$ has $k_i(\lambda)$ cycles of length $i$. Write $a(M,i)$ for $a(M,(i))$, where $\lambda=(i)$ corresponds to $w \in S_M$ where $G_w$ is connected and has a cycle of length $i$. For $G_w$ with cycle type $\lambda$, form a monomial
\[u_w=p_1^{k_1} p_2^{k_2} \dots p_n^{k_n},\]
where $p_j$ is a power sum symmetric function $p_j=\sum_{i \geq 1} x_i^j$. See Figure \ref{fig:figure1} for the digraph representation of a multiset permutation $w=(6, 1, 2, 4, 1, 9, 1, 6, 8, 4, 2, 10) \in S_{\{1, 1, 1, 2, 2, 4, 4, 6, 6, 8, 9, 10\}}$ with $u_w=p_1p_3$. Define the augmented cycle index for multiset permutations
$Z_{S_M}=\sum_{w \in S_M} u_w$.

\begin{figure}[h!]
\centering
    \begin{tikzpicture}
	    [thick, every node/.style={circle, draw=black!100, inner sep=1pt, minimum size=.6cm}, scale=1.25]
	    
	    \node (6) at (1,0) {$6$};
	    \node (1) at (0,.5) {$1$};
	    \node (2) at (-1,0) {$2$};
	    \node (9) at (1.75,.433) {$9$};
	    \node (8) at (1.75,-.433) {$8$};
	    \node (3) at (-1.75,.433) {$3$};
	    \node (11) at (-1.75,-.433) {$11$};
	    \node (7) at (.433,1.25) {$7$};
	    \node (5) at (-.433,1.25) {$5$};
	    \node (12) at (-3,1.25) {$12$};
	    \node (10) at (-3,0.5) {$10$};
	    \node (4) at (-3,-.25) {$4$};

	    \draw[<-] (4) to [out=310,in=230,looseness=7] (4);
	    \draw[->] (12) -- (10);
	    \draw[->] (10) -- (4);
	    \draw[->] (3) -- (2);
	    \draw[->] (11) -- (2);
	    \draw[->] (2) -- (1);
	    \draw[->] (1) -- (6);
	    \draw[->] (6) -- (9);
	    \draw[->] (9) -- (8);
	    \draw[->] (8) -- (6);
	    \draw[->] (5) -- (1);
	    \draw[->] (7) -- (1);
	    
	\end{tikzpicture}
\caption{The digraph of the multiset permutation $w=(6, 1, 2, 4, 1, 9, 1, 6, 8, 4, 2, 10)$ in $ S_{\{1, 1, 1, 2, 2, 4, 4, 6, 6, 8, 9, 10\}}$ which consists of one $1$-cycle (fixed point) $(4)$, one $3$-cycle $(6, 9, 8)$, and tree branches hanging from the cycles.}
\label{fig:figure1}
\end{figure}
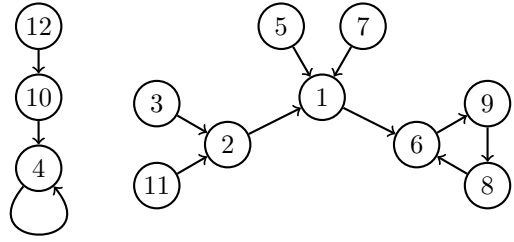

\begin{lemma}\label{lm:partition}
For $\lambda\vdash i$, we have
\begin{equation*}
a(M,\lambda) = i a(M,i)/z_\lambda,
\end{equation*}
where $z_\lambda=\prod_{j=1}^i j^{k_j} k_j!$ is as in Macdonald's book \cite[page 24]{Macdonald1995}.
\end{lemma}

\begin{proof}
The digraph $G_w$ corresponding to $a(M,i)$ comprises rooted trees arranged in a length-$i$ cycle. We keep the hanging tree branches intact with their incident vertices on the cycle and decompose the length-$i$ cycle into smaller cycles.
 
To do this, we write the length-$i$ cycle in one-line notation. There are $i$ possible starting points, which introduces an additional scalar factor $i$. Then we insert pairs of parentheses from left to right into this one-line notation according to $\lambda$, making sure that the lengths of the cycles are non-decreasing. This procedure generates the same digraph $z_\lambda$ times and we compensate for the overcounting by dividing by $z_\lambda$.
 
We note that our construction is an adaption from counting permutations of cycle type $\lambda$ in $S_i$, except that in our case the numerator becomes $i a(M,i)$ instead of $i!$.
\end{proof}

\begin{proposition}\label{pr:partition}
The augmented cycle index for multiset permutations satisfies
\[Z_{S_M}=\sum_{w \in S_M} u_w=\sum_{i=1}^n i a(M,i) h_i,\]
where $h_i=\sum_{1\leq j_1 \leq j_2 \leq \cdots \leq j_i} x_{j_1}x_{j_2}\cdots x_{j_i}$ is a complete symmetric function (the sum of all monomials of degree $i$).
\end{proposition}

\begin{proof}
By Corollary 7.7.6 in Stanley's book \cite{Stanley1999},
\[\sum_{i=1}^n i a(M,i) h_i =\sum_{i=1}^n i a(M,i) \left(\sum_{\lambda \vdash i} z_\lambda^{-1} p_\lambda \right),\]
where $z_\lambda$ is as defined earlier and $p_\lambda=p_{\lambda_1}p_{\lambda_2}\cdots$ is a power sum symmetric function for $\lambda=(\lambda_1, \lambda_2, \dots).$
The conclusion then follows from Lemma \ref{lm:partition}.
\end{proof}

\begin{example}
Take $M=\{1, 1, 2\}$ so that \[S_M=\{(1,1,2), (1,2,1), (2,1,1)\}.\] We compute that
\[Z_{S_M}=p_1+p_1^2+p_2=h_1+2h_2.\]
Take $M=\{1, 1, 2, 2\}$ so that \[S_M=\{(1,1,2,2), (1,2,1,2), (1,2,2,1), (2,1,1,2), (2,1,2,1), (2,2,1,1)\}.\] We compute that
\[Z_{S_M}=2p_1+2p_1^2+2p_2=2h_1+4h_2.\]
\end{example}

\begin{corollary}
\begin{equation*}
\sum_{i=1}^n i a(M,i) = |S_M|.
\end{equation*}
\end{corollary}

\begin{proof}
We take $x_1=1$ and $x_i=0$ for all $i\geq 2$, which yields that $p_i=h_i=1$ for all $i \geq 1$ and $u_w=1$ for all $w \in S_M$. By Proposition \ref{pr:partition}, the claimed identity is then immediate.
\end{proof}

\subsection{Terminal closers}
\label{subsec:tc}

Consider a partition $P$ of $[n]$ into parts $P_1, P_2, \ldots, P_n$, some of which may be empty.
In the FindStat database~\cite[St001050]{FindStat}, the {\em closers} of $P$ are defined to be those elements of $[n]$ which are maximal in their parts.
A closer $i$ is {\em terminal} if every $j \in \{i+1, \ldots, n\}$ is also a closer.
Terminal closers can also be interpreted in the context of multiset permutations.
Let $M$ be the $n$-element multiset over $[n]$ consisting of $|P_j|$ copies of $j$ for each $j$ in $[n]$, and define the multiset permutation $w$ of $M$ by
\begin{equation}\label{eq:bij-partition-closer}
    w_i = j \quad \mbox{if and only if} \quad i \in P_j.
\end{equation}
Notice that the elements $i, \ldots, n$ are terminal closers of $P$ if and only if $w_i \neq w_{i+1} \neq \cdots \neq w_n$.
Further, $\{i, \ldots, n\}$ is the set of {\em all} terminal closers of $P$ if and only if the number $i-1$ is not maximal in its part, which equates to having $w_{i-1} \in \{w_i, \ldots, w_n\}$.

For example, if $n = 6$, $P_1 = \{2,4\}$, $P_2 = \{1,3,6\}$, $P_3 = \{5\}$, and $P_4 = P_5 = P_6 = \varnothing$, then the closers are $4$, $5$, and $6$, and all are terminal.
Under the above correspondence $(P_1, \ldots, P_6) \mapsto (2,1,2,1,3,2)$, and the maximal end segment of distinct values is $(\ldots,1,3,2)$.
These are the labels of the parts $P_j$ containing terminal closers of the partition.
On the other hand, if we move the element $5$ from $P_3$ into $P_2$, then we obtain the partition $(\{2,4\}, \{1,3,5,6\}, \varnothing, \varnothing, \varnothing, \varnothing)$.
The closers here are $4$ and $6$, but only $6$ is a terminal closer.
Indeed, in the corresponding multiset permutation $(2,1,2,1,2,2)$, the maximal end segment of distinct values is $(\ldots, 2)$.

It would thus be natural to define the terminal closers of an arbitrary permutation $w$ of an arbitrary $n$-element multiset $M$ over $[n]$ to be the elements of the maximal end segment of distinct terms in $w$.
On the other hand, we could have instead defined the correspondence~\eqref{eq:bij-partition-closer} between multiset permutations and $n$-part partitions of $[n]$ by $w_{n + 1 - i} = j$ if and only if $i \in P_j$, in which case the natural analogue of the terminal closers of $P$ would be the elements of the maximal {\em initial} segment of distinct terms in $w$.
We choose the latter for the purposes of this paper, defining the set of {\em terminal closers of $w$} to be $\{w_1, \ldots, w_k\}$ where $w_1 \neq w_2 \neq \cdots \neq w_k$ but $w_{k+1}=w_i$ for some $i \in [k]$. 
(By default, $w$ has $n$ terminal closers if $w_1, \ldots, w_n$ are all distinct.) 

We proceed to describe a surprising correspondence between cyclic points and terminal closers for multiset permutations.
Notice that, if a multiset permutation $w \in S_M$ has exactly $k$ cyclic points with $k < n$, then $G_w$ is not simply a union of cycles. In particular, at least one cyclic point has in-degree at least $2$ in $G_w$. Using this observation and the commonplace tool of  Pr\"ufer code~{\cite{Prufer1918}}, we shall associate to $w$ a multiset permutation whose terminal closers are the cyclic points of $w$. We note that Pr\"ufer code and digraphs associated to multiset permutations (and mappings) have been used numerous times in the literature to provide combinatorial bijections, for instance, by Labelle and Gessel (see, e.g.,~{\cite{Gessel1995, Gessel2016}}) in proving various generalizations of the Lagrange inversion formula.

\begin{theorem}\label{thm:cyclic-terminal}
For any multiset $M$, the permutations of $M$ with a given set $S$ of cyclic points are in bijection with the permutations of $M$ whose set of terminal closers is $S$.
\end{theorem}

\begin{proof}
Let $M$ be an $n$-element multiset of positive integers taken from $[n]$, and let $w$ be an element of $S_M$, $w = (w_1, \ldots, w_n)$.
It is not hard to see that $w$ has at least one cycle, for the digraph $G_w$ associated to $w$ has $n$ edges and $n$ vertices.
Let $\{j_1, \ldots, j_k\}$ be the set of cyclic points of $G_w$, where $j_i < j_{i + 1}$ for each $i \in [k-1]$.
We construct an element $w'$ of $S_M$ whose terminal closers are $j_1, \ldots, j_k$ as follows.

\begin{enumerate}
    \item\label{item:cycles->terminals} For each $i \in [k]$, define $w'_i = w_{j_i}$.
    \item\label{item:Prufer} If $w$ is not a permutation, then $k < n$, and $G_w$ has a leaf.
    We define $w'_n, w'_{n-1}, \ldots, w'_{k+1}$ recursively.
    At step $i$, for $i \in \{0,1,\ldots, n-k-1\}$, let $\ell_i$ be the leaf of $G_w$ with the largest label, delete $\ell_i$ from $G_w - \{\ell_0, \ldots, \ell_{i-1}\}$, and define $w'_{n-i} = w_{\ell_i}$.
\end{enumerate}

Note that each one of $j_1, \ldots, j_k$ is a terminal closer of $w'$.
If $k < n$, then since $G_w - \{\ell_0, \ldots, \ell_{n-k-2}\}$ consists only of cycles and a single leaf $\ell_{n-k-1}$ adjacent to a cyclic point, we have $w_{\ell_{n-k-1}} = w'_{k+1} \in \{j_1, \ldots, j_k\}$.
Thus, $\{j_1, \ldots, j_k\}$ is precisely the set of terminal closers of $w'$.

The inspiration for item~\ref{item:Prufer} above comes from the Pr\"ufer code~\cite{Prufer1918}. That this construction is a bijection naturally follows.
\end{proof}

\begin{example}
    For the multiset permutation $w=(6, 1, 2, 4, 1, 9, 1, 6, 8, 4, 2, 10)$ with cyclic points $4$, $6$, $8$, and $9$, whose digraph $G_w$ is depicted in Figure~\ref{fig:figure1}, the permutation $w'$ having terminal closers $4,6,8,9$ obtained from the bijection in Theorem~\ref{thm:cyclic-terminal} is
    $(4,9,6,8,6,1,2,1,1,4,2,10)$.
    The digraph $G_{w'}$ is depicted in Figure~\ref{fig:w'}.

    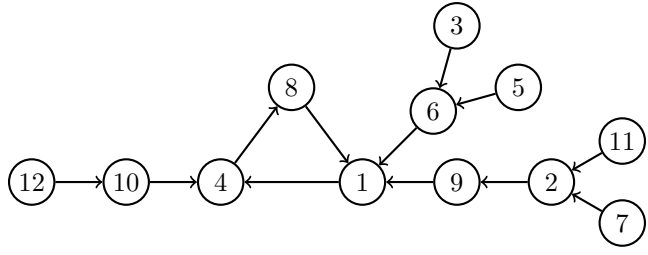
\begin{figure}
        \centering
        \begin{tikzpicture}
	    [thick, every node/.style={circle, draw=black!100, inner sep=1pt, minimum size=.6cm}, scale=1.25]
	    
	    \node (12) at (-2.75,0) {$12$};
	    \node (10) at (-1.75,0) {$10$};
	    \node (4) at (-.75,0) {$4$};
	    \node (8) at (0,1) {$8$};
	    \node (1) at (.75,0) {$1$};
	    \node (9) at (1.75,0) {$9$};
	    \node (2) at (2.75,0) {$2$};
	    \node (11) at (3.5,.433) {$11$};
	    \node (7) at (3.5,-.433) {$7$};
	    \node (6) at (1.5,.75) {$6$};
	    \node (3) at (1.75,1.65) {$3$};
	    \node (5) at (2.4,1) {$5$};

	    \draw[->] (1) -- (4);
	    \draw[->] (12) -- (10);
	    \draw[->] (10) -- (4);
	    \draw[->] (7) -- (2);
	    \draw[->] (11) -- (2);
	    \draw[->] (6) -- (1);
	    \draw[->] (3) -- (6);
	    \draw[->] (2) -- (9);
	    \draw[->] (4) -- (8);
	    \draw[->] (5) -- (6);
	    \draw[->] (8) -- (1);
	    \draw[->] (9) -- (1);
	    \end{tikzpicture}
        
        \caption{The digraph of the multiset permutation $w' = (4,9,6,8,6,1,2,1,1,4,2,10)$ in $S_{\{1, 1, 1, 2, 2, 4, 4, 6, 6, 8, 9, 10\}}$ obtained from the multiset permutation $w \in S_{\{1, 1, 1, 2, 2, 4, 4, 6, 6, 8, 9, 10\}}$ (whose digraph is depicted in Figure~\ref{fig:figure1}) via the bijection in Theorem~\ref{thm:cyclic-terminal}.
        The set $\{4,6,8,9\}$ of terminal closers of $w'$ is the set of cyclic points of $w$.}
        \label{fig:w'}
    \end{figure}
\end{example}

\begin{corollary}\label{cor:t(M,k)-and-a(M,k)}
	Let $M$ be a multiset. Let $a(M, k)$ denote the number of permutations $w$ of the multiset $M$ for which $G_w$ is connected and has a cycle of length $k$. Let $t(M,k)$ denote the number of permutations of $M$ with exactly $k$ terminal closers. Then
	$a(M,k) = t(M,k) / k$.
\end{corollary}

\begin{proof}
The claim readily follows once we note that exactly $(k-1)!$ of the $k!$ permutations of $S$ consist of a single cycle, as opposed to a disjoint union of at least two cycles.
\end{proof}

Let $M$ be a multiset, which is denoted by $\bigl(m_1^{n_1}, \ldots, m_{j_M}^{n_{j_M}}\bigr)$, where $j_M$ is the number of distinct parts of $M$, $1 \leq m_1 < \cdots < m_{j_M} \leq n$ and $\sum_i n_i = n$. Corollary \ref{cor:t(M,k)-and-a(M,k)} provides an alternative (and faster) way to calculate $a(M, k)$, as only the values of the initial $k+1$ terms in $w \in S_M$ matter. We present two special cases as examples.

\begin{example}[{$a(M,1) = t(M,1)$}]
When $k = 1$, for a permutation $w = (w_1, \ldots, w_n)$ of $M$ to have a single terminal closer, we must have some $n_s \geq 2$ so that $w_1 = w_2$.
Following $w_1$ and $w_2$, we can have any of the $(n-2)! / (\prod_i n_i'!)$ permutations of the remaining elements of $M$, where $n_i' = n_i$ for $i \neq s$ and $n_s' = n_s - 2$.
Summing over all $s$ such that $n_s \geq 2$, we obtain $t(M,1)$.
For instance, when $M=\{1,1,1,1,2,2,2,2,3,3\} = \bigl( 1^{4}, 2^{4}, 3^{2} \bigr)$, we obtain
\[
t(M,1) = \frac{8!}{2! \cdot 4! \cdot 2!} + \frac{8!}{4! \cdot 2! \cdot 2!} + \frac{8!}{4! \cdot 4!} = 910.
\]
\end{example}

\begin{example}[{$a(M,j_M)$}]\label{ex}
	If every $m_i$ in $M$ appears as a terminal closer, then we have $j_M!$ ways to permute the terminal closers and can choose any permutation of $\bigl( m_1^{n_1 - 1}, \ldots, m_{j_M}^{n_{j_M} - 1} \bigr)$ to follow. Thus,
	\[
	t(M,j_M) = j_M ! \frac{(n-j_M)!}{\prod_{i} (n_i - 1)!}, \qquad \mbox{and} \qquad a(M,j_M) = \frac{(n-j_M)! (j_M - 1)!}{\prod_i (n_i - 1)!}.
	\]
\end{example}

\section{Mappings and permutations}
\label{sec:mp}

Any set of permutation invariant mappings is a union of certain $S_M$'s, where each $M$ is an $n$-element multiset of positive integers taken from $[n]$. In particular, $\mathcal{F}_n$ is the union of $S_M$ for all $n$-element multisets $M$ of $[n]$.
Permutations are trivially permutation invariant and $S_n=S_M$ when the $n$-element multiset $M$ exactly equals $[n]$. Note that $S_n \subseteq \mathcal{F}_n$. Given a mapping $f= (f_1, \dots, f_n) \in \mathcal{F}_n$, let $G_f$ denote the digraph on the vertex set $[n]$ with edges $i \rightarrow f_i$.

Mappings and permutations have been well studied. We present some standard results in this section.

\begin{theorem}[{see~\cite[section 3.3.13]{Goulden1983}}]\label{main-cycle-count: mapping}
The number of mappings $f$ of length $n$ where $G_f$ has $k_i$ cycles of length $i$ for every $i \in [n]$ is given by
\begin{equation*}
\frac{1}{\prod_{i=1}^n i^{k_i} k_i!}\binom{n}{k} k! \cdot k n^{n-k-1},
\end{equation*}
where $k=\sum_{i=1}^n k_i i$ is the number of cyclic points with $1\leq k\leq n$.
\end{theorem}

\begin{theorem}\label{main-cycle-count-special: mapping}
Take $1\leq k\leq n$. The number of mappings $f$ of length $n$ for which $G_f$ is connected and has a cycle of length $k$ is given by
\begin{equation*}
A(n, k)=k! \binom{n}{k} n^{n-k-1}.
\end{equation*}
\end{theorem}

\begin{proof}
Theorem \ref{main-cycle-count-special: mapping} is a special case of Theorem \ref{main-cycle-count: mapping} where we take $(k_1, \dots, k_n)$ to be the elementary vector $e_k$ that is $1$ in the $k$th position and $0$ elsewhere.
\end{proof}

\begin{remark}\label{rk:mapping}
Let $C(n)=\sum_{k=1}^n A(n, k)$. Then $C(n)$ counts the number of mappings $f$ of length $n$ for which $G_f$ is connected. The first few terms of $\{C(n)\}_{n \geq 1}$ are given by
\[1, 3, 17, 142, 1569, 21576, 355081, 6805296,\]
which is \cite[OEIS A001865]{OEIS}.
\end{remark}

Following the notation in Theorem \ref{main-cycle-count-special: mapping}, $A(n,i)$ is the number of mappings $f$ of length $n$ such that $G_f$ is connected 
and has a cycle of length $i$. For $G_f$ with cycle type $\lambda=(1^{k_1(\lambda)}, 2^{k_2(\lambda)}, \dots)$, form a monomial 
\[u_f=p_1^{k_1} p_2^{k_2} \dots p_n^{k_n},\]
where $p_j$ is a power sum symmetric function. Define the augmented cycle index for mappings
\begin{equation*}
Z_n=\sum_{f \in \mathcal{F}_n} u_f,
\end{equation*}
where the sum is over all mappings of length $n$, so $Z_n$ is an inhomogeneous symmetric function and $A(n,i)$ is the coefficient of $p_i$ in $Z_n$.

\begin{proposition}\label{pr:mapping-augmented}
The augmented cycle index for mappings satisfies
\[Z_n=\sum_{f \in \mathcal{F}_n} u_f=\sum_{i=1}^n i A(n,i) h_i,\]
where $h_i$ is a complete symmetric function (the sum of all monomials of degree $i$).
\end{proposition}

\begin{proof}
Note that mappings are permutation invariant and thus $\mathcal{F}_n$ is a union of certain $S_M$'s. The result then follows readily from Proposition \ref{pr:partition}.
\end{proof}

\begin{example}
Take $n=3$. We compute that
\[Z_3=9p_1+6p_2+2p_3+6p_1^2+3p_1p_2+p_1^3=9h_1+12h_2+6h_3.\]
\end{example}

\begin{corollary}
\begin{equation*}
\sum_{i=1}^n i A(n,i) = n^n.
\end{equation*}
\end{corollary}

\begin{proof}
We take $x_1=1$ and $x_i=0$ for all $i\geq 2$, which yields that $p_i=h_i=1$ for all $i \geq 1$ and $u_f=1$ for all $f \in \mathcal{F}_n$. By Proposition \ref{pr:mapping-augmented}, the claimed identity is then immediate.
\end{proof}

Since permutations are mappings with exactly $n$ cyclic points, the following results may be seen as special cases of Theorems \ref{main-cycle-count: mapping} and \ref{main-cycle-count-special: mapping} and Proposition \ref{pr:mapping-augmented}.

\begin{theorem}\label{main-cycle-count: permutation}
The number of permutations $\sigma$ of length $n$ where $G_\sigma$ has $k_i$ cycles of length $i$ for every $i \in [n]$ is given by
\begin{equation*}
\frac{1}{\prod_{i=1}^n i^{k_i} k_i!}n!\,.
\end{equation*}
\end{theorem}

\begin{theorem}\label{main-cycle-count-special: permutation}
The number of permutations $\sigma$ of length $n$ for which $G_\sigma$ is connected (and thus has a cycle of length $n$) is given by
\begin{equation*}
A(n, n)=(n-1)!\,.
\end{equation*}
\end{theorem}

\begin{proposition}\label{pr:permutation-augmented}
The augmented cycle index for permutations satisfies
\[Z_n=\sum_{\sigma \in S_n} u_\sigma=n! h_n,\]
where $h_n$ is a complete symmetric function (the sum of all monomials of degree $n$).
\end{proposition}

\begin{example}
Take $n=3$. We compute that
\[Z_3=2p_3+3p_1p_2+p_1^3=6h_3.\]
\end{example}

\section{Parking functions}
\label{sec:parking}
Lemma \ref{lm:partition} and Proposition \ref{pr:partition} provide the core argument for the augmented cycle index decomposition for any union of multiset permutations; mappings and permutations only serve as special cases.

In this section, we study a combinatorial object that lies between mappings and permutations. This combinatorial object is commonly referred to as a parking function and was introduced by Konheim and Weiss \cite{Konheim1966} in their study of the hash storage structure.

In the classical parking
function scenario, we have $n$
parking spaces on a one-way street, labelled $1,2,\dots,n$ in
consecutive order as we drive down the street. There are $n$ cars
$C_1,\dots,C_n$. Each car $C_i$ has a preferred space $1\leq \pi_i\leq
n$. The cars drive down the street one at a time in order
$C_1,\dots, C_n$. The car $C_i$ drives immediately to space $\pi_i$ and
then parks in the first available space. Thus if $\pi_i$ is empty, then
$C_i$ parks there; otherwise $C_i$ next goes to space $\pi_i+1$, $\pi_i+2$, etc., until it finds an available space to park in (if no such space exists, then $C_i$ leaves the street unparked). If
all cars are able to park, then the sequence $\pi=(\pi_1,\dots, \pi_n)$ is
called a \emph{parking function} of length $n$.

Write $\PF_n$
for the set of parking functions of length $n$, so that $|\PF_n|=(n+1)^{n-1}$. An elegant, unpublished proof of this result using a circular symmetry argument was given by Pollak and is recounted in \cite{Foata1974} and \cite{Riordan1969}. Observe that every permutation is a parking function and every parking function is a mapping, so that $S_n \subseteq \PF_n \subseteq \mathcal{F}_n$. 

\begin{theorem}\label{main-cycle-count}
The number of parking functions $\pi$ of length $n$ where $G_\pi$ has $k_i$ cycles of length $i$ for every $i \in [n]$ is given by
\begin{equation*}
\begin{cases}
\frac{1}{\prod_{i=1}^n i^{k_i} k_i!} n! & \text{ if } k = n, \\
\frac{1}{\prod_{i=1}^n i^{k_i} k_i!}\binom{n+1}{k} k! \cdot k(n+1)^{n-k-2} & \text{ if } 1 \leq k < n,
\end{cases}
\end{equation*}
where $k=\sum_{i=1}^n k_i i$ is the number of cyclic points.
\end{theorem}

\begin{remark}
We note that a parking function $\pi$ of length $n$ has $n$ cyclic points if and only if the parking preferences constitute a permutation of $1, \dots, n$. Thus in this special case, our formula coincides with the number of permutations in $S_n$ of cycle type $(1^{k_1}, 2^{k_2}, \dots)$.
\end{remark}

Given $\pi=(\pi_1, \dots, \pi_n) \in \PF_n$, let $G_\pi$ denote the digraph on the vertex set $[n]$ with edges $i \rightarrow \pi_i$. Let $a(n,\lambda)$ be the number of $\pi\in \PF_n$ such that $G_\pi$ has cycle type 
$\lambda=(1^{k_1(\lambda)}, 2^{k_2(\lambda)}, \dots)$, i.e., $G_\pi$ has $k_i(\lambda)$ cycles of length $i$. Write $a(n,i)$ for $a(n,(i))$.

\begin{lemma}\label{lem:a(n,lambda)-and-a(n,i)}
For $\lambda\vdash i$, we have
\begin{equation*}\label{eq:identity}
a(n,\lambda) = i a(n,i)/z_\lambda,
\end{equation*}
where $z_\lambda=\prod_{j=1}^i j^{k_j} k_j!$ is as defined earlier.
\end{lemma}

\begin{proof}
Note that parking functions are permutation invariant and thus $\PF_n$ is a union of certain $S_M$'s. The result then follows readily from Lemma \ref{lm:partition}.
\end{proof}

\begin{proof}[Proof of Theorem \ref{main-cycle-count}]
Let $M$ be a multiset for which $S_M \subseteq \PF_n$.
By Corollary~\ref{cor:t(M,k)-and-a(M,k)}, we have $t(M,k) = k a(M,k)$.
Thus, the number $t(n,k)$ of parking functions in $\PF_n$ with exactly $k$ terminal closers is $k a(n,i)$.
Lemma~\ref{lem:a(n,lambda)-and-a(n,i)} then implies that, for any partition $\lambda = (1^{k_1}, 2^{k_2}, \ldots)$ of $k$, we have 
\[
t(n,k) = z_\lambda a(n,\lambda),
\]
where $z_\lambda=\prod_{i=1}^k j^{k_i} k_i!$ is as defined earlier.

The case $k = n$ is precisely Theorem~\ref{main-cycle-count-special: permutation}, so we assume that $k < n$.
Using a variation of Pollak's circular symmetry argument \cite{Foata1974, Riordan1969}, we are able to count $t(n,k)$ relatively easily.
Consider $n$ cars $C_1, \ldots, C_n$ parking in spaces numbered $1, 2, \ldots, n+1$ on a circular one-way street.
Suppose that cars $C_1, \ldots, C_k$ have distinct preferences, and that car $C_{k+1}$ prefers one of these spots.
Cars $C_{k+2}, \ldots, C_n$ may have any preferences in $[n+1]$.
Of course, all cars will be able to park, and $C_1, \ldots, C_k$ will park in their desired spots.
Of the $n+1$ circular rotations of the labels of the parking spots, exactly one corresponds to a parking function $\pi$: the one which leaves spot $n+1$ open.
If $C_1, \ldots, C_k$ prefer spots $\pi_1, \ldots, \pi_k$ in $\pi$, then these are the terminal closers of $\pi$, and any set of $k$ terminal closers can be obtained in this manner.
Thus,
\[
t(n,k) = \frac{k! \binom{n+1}{k} k (n+1)^{n + 1 - k - 2} }{n+1} = k! \binom{n+1}{k} k (n+1)^{n - k - 2}.
\]
The statement follows.
\end{proof}

\begin{theorem}\label{main-cycle-count-special}
The number of parking functions $\pi$ of length $n$ for which $G_\pi$ is connected and has a cycle of length $k$ is given by
\begin{equation*}
a(n, k)=\begin{cases}
(n-1)! & \text{ if } k = n, \\
k! \binom{n+1}{k} (n+1)^{n-k-2} & \text{ if } 1 \leq k < n.
\end{cases}
\end{equation*}
\end{theorem}

\begin{remark}
We note that the digraph $G_{\pi}$ associated to a parking function $\pi$ of length $n$ consists of a single cycle of length $n$ if and only if the parking preferences constitute a cycle of $1, \dots, n$. Thus in this special case, our formula coincides with the number of permutations in $S_n$ which form an $n$-cycle.
\end{remark}

\begin{remark}\label{rk:parking}
Let $c(n)=\sum_{k=1}^n a(n, k)$. Then $c(n)$ counts the number of parking functions $\pi$ of length $n$ for which $G_\pi$ is connected. The first few terms of $\{c(n)\}_{n \geq 1}$ are given by
\[1, 2, 9, 63, 600, 7249, 106344, 1837953,\]
which is not in the OEIS \cite{OEIS}.
\end{remark}

\begin{proof}[Proof of Theorem \ref{main-cycle-count-special}]
Theorem \ref{main-cycle-count-special} is a special case of Theorem \ref{main-cycle-count} where we take $(k_1, \dots, k_n)$ to be the elementary vector $e_k$ that is $1$ in the $k$th position and $0$ elsewhere.
\end{proof}

For $G_\pi$ with cycle type $\lambda=(1^{k_1(\lambda)}, 2^{k_2(\lambda)}, \dots)$, form a monomial \[u_\pi=p_1^{k_1} p_2^{k_2} \dots p_n^{k_n},\]
where $p_j$ is a power sum symmetric function. Define the augmented cycle index for parking functions
\begin{equation*}
Z_n=\sum_{\pi \in \PF_n} u_\pi,
\end{equation*}
where the sum is over all parking functions of length $n$, so $Z_n$ is an inhomogeneous symmetric function and $a(n,i)$ is the coefficient of $p_i$ in $Z_n$.

\begin{proposition}\label{pr:parking-augmented}
The augmented cycle index for parking functions satisfies
\[Z_n=\sum_{\pi \in \PF_n} u_\pi=\sum_{i=1}^n i a(n,i) h_i,\]
where $h_i$ is a complete symmetric function (the sum of all monomials of degree $i$).
\end{proposition}

\begin{proof}
Note that parking functions are permutation invariant and thus $\PF_n$ is a union of certain $S_M$'s. The result then follows readily from Proposition \ref{pr:partition}.
\end{proof}

\begin{example}
Take $n=3$. We compute that
\[Z_3=4p_1 + 3p_2 + 2p_3 +3p_1^2 + 3p_2 p_1 + p_1^3=4h_1+6h_2+6h_3.\]
\end{example}

\begin{corollary}
\begin{equation*}
\sum_{i=1}^n i a(n,i) = (n+1)^{n-1}.
\end{equation*}
\end{corollary}

\begin{proof}
We take $x_1=1$ and $x_i=0$ for all $i\geq 2$, which yields that $p_i=h_i=1$ for all $i \geq 1$ and $u_\pi=1$ for all $\pi \in \PF_n$. By Proposition \ref{pr:parking-augmented}, the claimed identity is then immediate.
\end{proof}

\section{Prime parking functions}
\label{sec:prime_parking}
In this section, we focus on a subclass of parking functions called {\em prime parking functions}. A parking function $\pi=(\pi_1, \ldots, \pi_n)$ is a prime parking function if for all $1 \leq j \leq n-1$, at least $j+1$ cars have a parking preference in the first $j$ parking spots. Equivalently, $\pi$ is a prime parking function if it remains a parking function even after removing a coordinate that equals $1$. Let $\PPF_n$ be the set of prime parking functions of length $n$. Kalikow (see \cite[pages 141--142]{Stanley1999}) modified Pollak's circular symmetry argument to show that
    $|\PPF_n| = (n-1)^{n-1}$, so that $|\PPF_n| = |\mathcal{F}_{n-1}|$. 
Observe that $\PPF_n$ and $S_n$ are disjoint subsets of $\PF_n$, that is, $\PPF_n \subseteq \PF_n$ and $S_n \subseteq \PF_n$, but $\PPF_n \cap \ S_n =\emptyset$.

\begin{theorem}\label{main-cycle-count: prime}
The number of prime parking functions $\pi$ of length $n$ where $G_\pi$ has $k_i$ cycles of length $i$ for every $i \in [n-1]$ is given by
\begin{equation*}
\frac{1}{\prod_{i=1}^{n-1} i^{k_i} k_i!}\binom{n-1}{k} k! k(n-1)^{n-k-2},
\end{equation*}
where $k=\sum_{i=1}^{n-1} k_i i$ is the number of cyclic points with $1\leq k \leq n-1$.
\end{theorem}

\begin{remark}
We note that a prime parking function $\pi$ of length $n$ cannot have $n$ cyclic points, since by definition $\pi$ must have at least two coordinates equalling $1$ and no coordinate equalling $n$.
\end{remark}

\begin{remark}
We note that the number of mappings $f$ of length $n-1$ where $G_f$ has $k_i$ cycles of length $i$ for every $i \in [n-1]$ coincides with the number of prime parking functions $\pi$ of length $n$ where $G_\pi$ has $k_i$ cycles of length $i$ for every $i \in [n-1]$.
\end{remark}

Recall that $G_\pi$ denotes the digraph on the vertex set $[n]$ with edges $i \rightarrow \pi_i$. Let $a'(n,\lambda)$ be the number of $\pi\in \PPF_n$ such that $G_\pi$ has cycle type 
$\lambda=(1^{k_1(\lambda)}, 2^{k_2(\lambda)}, \dots)$, i.e., $G_\pi$ has $k_i(\lambda)$ cycles of length $i$. Write $a'(n,i)$ for $a'(n,(i))$. 

\begin{lemma}\label{lem:a'(n,lambda)-and-a'(n,i)}
For $\lambda\vdash i$, we have
\begin{equation*}\label{eq:identity'}
a'(n,\lambda) = i a'(n,i)/z_\lambda,
\end{equation*}
where $z_\lambda=\prod_{j=1}^i j^{k_j} k_j!$ is as defined earlier.
\end{lemma}

\begin{proof}
Note that prime parking functions are permutation invariant and thus $\PPF_n$ is a union of certain $S_M$'s. The result then follows readily from Lemma \ref{lm:partition}.
\end{proof}

\begin{proof}[Proof of Theorem \ref{main-cycle-count: prime}]
The proof proceeds similarly as the proof of Theorem \ref{main-cycle-count} for parking functions, except that our circular symmetry argument is now applied to a circle with $n-1$ spots instead of $n+1$ spots as in the classical case.
\end{proof}

\begin{theorem}\label{main-cycle-count-special: prime}
Take $1\leq k \leq n-1$. The number of prime parking functions $\pi$ of length $n$ for which $G_\pi$ is connected and has a cycle of length $k$ is given by
\begin{equation*}
a'(n, k)=k! \binom{n-1}{k} (n-1)^{n-k-2}.
\end{equation*}
\end{theorem}

\begin{remark}\label{rk:prime_parking}
Let $c'(n)=\sum_{k=1}^{n-1} a'(n, k)$. Then $c'(n)$ counts the number of prime parking functions $\pi$ of length $n$ for which $G_\pi$ is connected. The first few terms of $\{c'(n)\}_{n \geq 1}$ are given by
\[0, 1, 3, 17, 142, 1569, 21576, 355081,\]
which by shifting one term is \cite[OEIS A001865]{OEIS}. We note that the number of mappings $f$ of length $n-1$ for which $G_f$ is connected coincides with the number of prime parking functions $\pi$ of length $n$ for which $G_\pi$ is connected.
\end{remark}

\begin{proof}[Proof of Theorem \ref{main-cycle-count-special: prime}]
Theorem \ref{main-cycle-count-special: prime} is a special case of Theorem \ref{main-cycle-count: prime} where we take $(k_1, \dots, k_n)$ to be the elementary vector $e_k$ that is $1$ in the $k$th position and $0$ elsewhere.
\end{proof}

For $G_\pi$ with cycle type $\lambda=(1^{k_1(\lambda)}, 2^{k_2(\lambda)}, \dots)$, form a monomial \[u_\pi=p_1^{k_1} p_2^{k_2} \dots p_n^{k_n},\]
where $p_j$ is a power sum symmetric function. Define the augmented cycle index for prime parking functions
\begin{equation*}
Z_n=\sum_{\pi \in \PPF_n} u_\pi,
\end{equation*}
where the sum is over all prime parking functions of length $n$, so $Z_n$ is an inhomogeneous symmetric function and $a(n,i)$ is the coefficient of $p_i$ in $Z_n$.

\begin{proposition}\label{pr:parking-augmented: prime}
The augmented cycle index for prime parking functions satisfies
\[Z_n=\sum_{\pi \in \PPF_n} u_\pi=\sum_{i=1}^n i a'(n,i) h_i,\]
where $h_i$ is a complete symmetric function (the sum of all monomials of degree $i$).
\end{proposition}

\begin{proof}
Note that parking functions are permutation invariant and thus $\PPF_n$ is a union of certain $S_M$'s. The result then follows readily from Proposition \ref{pr:partition}.
\end{proof}

\begin{example}
Take $n=3$. We compute that
\[Z_3=2p_1 +p_2+p_1^2=2h_1+2h_2.\]
\end{example}

\begin{corollary}
\begin{equation*}
\sum_{i=1}^n i a'(n,i) = (n-1)^{n-1}.
\end{equation*}
\end{corollary}

\begin{proof}
We take $x_1=1$ and $x_i=0$ for all $i\geq 2$, which yields that $p_i=h_i=1$ for all $i \geq 1$ and $u_\pi=1$ for all $\pi \in \PPF_n$. By Proposition \ref{pr:parking-augmented: prime}, the claimed identity is then immediate.
\end{proof}

\section{Asymptotic analysis}
\label{sec:aa}
For convenience, we say that a multiset permutation (mapping, parking function, prime parking function, etc.) is \emph{indecomposable (connected)} if its digraph is connected, i.e., the digraph possesses exactly one component. Recall that $C(n)=\sum_{k=1}^n A(n,k)$ counts the total number of connected mappings (cf.~Remark~\ref{rk:mapping}). Katz \cite{Katz1955} established that the probability that a random mapping is indecomposable is asymptotically equal to $\sqrt{\pi/(2n)}$, that is, $C(n)/n^n \sim \sqrt{\pi/(2n)}$. R\'{e}nyi \cite{Renyi1959} further showed that the number of cyclic points (equivalently, the length of the unique cycle), scaled by $\sqrt{n}$, of an indecomposable random mapping converges in distribution to a folded normal distribution $|N(0,1)|$. See also related discussions in the paper of Mutafchiev and Finch \cite{Mutafchiev2024}.

A unifying theme in the probabilistic study of parking functions is the notion of the {\em equivalence of ensembles}.
Diaconis and Hicks \cite{Diaconis2017} stated that for certain statistics, it is natural to expect the distribution of statistics in the \emph{micro-canonical ensemble} $\PF_n$, to be close to the distribution of statistics in the \emph{canonical ensemble} $\mathcal{F}_n$. Indeed, they showed that the equivalence of ensembles holds for statistics such as the number of repeats, lucky cars, and descents. However, they also showed that it fails for some statistics such as the distribution of the value of the first coordinate.

In this section, we will demonstrate the asymptotic equivalence of ensembles between parking functions and mappings concerning the number of cyclic points in their associated connected digraphs, in the sense that both the probability of a random parking function being connected is asymptotically equal to $\sqrt{\pi/(2n)}$ as in \cite{Katz1955}, and that the number of cyclic points in a random connected parking function, scaled by $n^{-1/2}$, converges in distribution to $|N(0,1)|$ as in \cite{Renyi1959}. The same also holds true for connected prime parking functions. Finally, we will extend these techniques to connected regular multiset permutations. 

We note that this asymptotic equivalence of ensembles between parking functions and mappings concerning cyclic points does not come as a surprise. Indeed, parking functions and mappings behave similarly not just asymptotically but also combinatorially. Since parking functions and mappings are both unions of multiset permutations, a direct consequence of Theorem \ref{thm:cyclic-terminal} is that the number of parking functions/mappings of length $n$ with $k$ cyclic points respectively equals the number of parking functions/mappings of length $n$ with $k$ terminal closers. The number of parking functions of length $n$ with $k$ terminal closers was computed in the proof of Theorem \ref{main-cycle-count} using a variation of Pollak's circular symmetry argument and was shown to be $k! \binom{n+1}{k} k (n+1)^{n - k - 2}$. Following the definition for terminal closers, the number of mappings of length $n$ with $k$ terminal closers may be easily calculated as $k! \binom{n}{k} k n^{n-k-1}$ (see also Theorem \ref{main-cycle-count: mapping}, where we cited a related result from standard literature). We see that the two numbers and the way they are obtained display striking similarity.

We remark that the natural analogous question for the number of cyclic points in a uniformly random parking function has been answered in \cite{Paguyo2025}, where the authors showed that this number is asymptotically Rayleigh-distributed after rescaling by $n^{-1/2}$.

Throughout this section, we use $\Poi(\lambda)$ to denote a Poisson-distributed random variable with parameter $\lambda>0$, and $\NB(\ell,p)$ to denote a negative binomial random variable, i.e. the random variable counting the number of failures in a sequence of i.i.d. Bernoulli-trials with success probability $p$ before seeing the $\ell$-th success. We note here the very well-known central limit theorems
\[
 \frac{\Poi(n)-n}{\sqrt{n}} \dto N(0,1) \qquad \text{ and } \qquad \frac{\NB(\ell,p)-\ell(1-p)/p}{\sqrt{\ell(1-p)/p^2}} \dto N(0,1)
\]
for $n\to\infty$ and $\ell\to\infty$ with $p\in(0,1)$ constant, respectively. 
Both of these statements can easily be obtained after noting that $\Poi(n)$ is the sum of $n$ i.i.d. $\Poi(1)$ random variables, and $\NB(\ell,p)$ is the sum of $\ell$ i.i.d. geometric random variables with parameter $p$.

\subsection{Connected parking functions}\label{subsec:cpf}
Recall that $c(n)=\sum_{k=1}^n a(n,k)$ counts the total number of connected parking functions (cf. Remark \ref{rk:parking}). 

\begin{lemma}\label{lm:parking}
 We have $c(n) \sim \sqrt{\pi/2} \cdot e n^{n-3/2}$, so that $c(n)/(n+1)^{n-1} \sim \sqrt{\pi/(2n)}$. 
\end{lemma}

\begin{proof}
 We first consider $S_n=\sum_{k=1}^{n-1} a(n,k)$. By rewriting the above expression for $a(n,k)$, we obtain
 \begin{align}\label{eq:cnsetup}
  S_n 
  &= \frac{(n+1)!}{(n+1)^3} \sum_{k=1}^{n-1} \frac{(n+1)^{n+1-k}}{(n+1-k)!}
   = \frac{(n+1)!}{(n+1)^3} \sum_{j=2}^n \frac{(n+1)^j}{j!}\notag\\
  &= \frac{(n+1)!}{(n+1)^3} e^{n+1} \IP\big(2\leq \Poi(n+1)\leq n\big).
 \end{align}
 By the central limit theorem, the probability converges to $1/2$; combining this with Stirling's formula for $(n+1)!$ yields 
 \[
  S_n
  \sim \frac{\sqrt{2\pi}}{2} (n+1)^{n+3/2-3}, 
 \]
 which simplifies to the desired expression. It remains to show that $a(n,n)=o(S_n)$. For this, note that
 \[
  \frac{a(n,n)}{S_n}
  \sim \frac{\sqrt{2\pi}(n-1)^{n-1/2} e^{-(n-1)}}{\sqrt{\pi/2}(n+1)^{n-3/2}} 
  = \frac{2(n-1)}{e^{n-1}} \left(1-\frac{2}{n+1}\right)^{n-3/2}
  \sim \frac{2(n-1)}{e^{n+1}} \to 0,
 \]
 therefore $c(n)\sim S_n$, as required. 
\end{proof}

\begin{remark}
 The sum in \eqref{eq:cnsetup} is well-studied: the Ramanujan $Q$-function is defined by
 $Q(n) = \frac{n!}{n^n} \sum_{k=0}^{n-1} \frac{n^k}{k!}$, which can also be written as $\frac{n! e^n}{n^n} \IP\big( \Poi(n)\leq n-1\big)$. We have $S_n = (n+1)^{n-2}Q(n+1)
 - \frac{(n+2)!}{(n+1)^3}$, and it is known that
   \[Q(n) = \sqrt{\frac{\pi n}{2}} - \frac13 + \frac1{12} \sqrt{\frac{\pi}{2n}} - \frac{4}{135n} + O(n^{-3/2}).\]
The Ramanujan $Q$-function is a special case of the more general $Q_r$-function, where
\[Q_r(m, n)=\binom{r}{0}+\binom{r+1}{1} \frac{n}{m}+\binom{r+2}{2} \frac{n(n-1)}{m^2}+\cdots\]
is a generalized hypergeometric function of the second kind, and $Q(n)=Q_0(n, n-1)$. See \cite{Flajolet1995} and \cite{Flajolet1998}.
\end{remark}

Building upon Lemma \ref{lm:parking}, we can now determine the limiting distribution of $L_n$, the number of cyclic points in a uniform random connected parking function.

\begin{theorem}\label{thm:parking}
 We have $n^{-1/2}L_n\dto |N(0,1)|$ as $n\to\infty$. 
\end{theorem}
\begin{proof}
 Fix $x>0$ and set $R=R_n=\floor{x\sqrt{n}}$. Consider $n$ large enough so that $R<n$. Similarly to \eqref{eq:cnsetup}, we compute 
 \begin{align}\label{eq:foldedproof}
  \IP(L_n\leq R)
  &= \frac{1}{c(n)} \sum_{k=1}^R a(n,k)
   = \frac{1}{c(n)} \frac{(n+1)!}{(n+1)^3} e^{n+1} \IP\big( n+1-R \leq \Poi(n+1)\leq n \big)\notag\\
  &\sim 2 \int_{-\frac{R}{\sqrt{n+1}}}^{-\frac{1}{\sqrt{n+1}}} \frac{1}{\sqrt{2\pi}} e^{-y^2/2} \Di y
   \sim \int_{-x}^x \frac{1}{\sqrt{2\pi}}  e^{-y^2/2} \Di y,
 \end{align}
 where we again relied on the central limit theorem, using $c(n)\sim \frac{(n+1)!}{(n+1)^3} e^{n+1}/2$ as a consequence of \eqref{eq:cnsetup}. Noting that the left-hand side of \eqref{eq:foldedproof} equals $\IP(n^{-1/2}L_n \leq x)$ while the right-hand side is the cumulative distribution function of $|N(0,1)|$ concludes the proof. 
\end{proof}

Recall that $c'(n)=\sum_{k=1}^{n-1} a'(n,k)$ counts the total number of connected prime parking functions (cf. Remark \ref{rk:prime_parking}). Let $L'_n$ be the number of cyclic points in a uniform random connected prime parking function.

\begin{proposition}
 We have $c'(n)\sim \sqrt{\pi/2} e^{-1} n^{n-3/2}$ (so that $c'(n)/(n-1)^{n-1} \sim \sqrt{\pi/(2n)}$) and $n^{-1/2}L'_n\dto |N(0,1)|$.
\end{proposition}
\begin{proof}
 The proof works analogously to those of Lemma \ref{lm:parking} and Theorem \ref{thm:parking}. Indeed, imitating the approach of \eqref{eq:cnsetup}, together with the central limit theorem and a Stirling approximation yields
 \begin{align*}
  c'(n)
  &= (n-2)!\sum_{k=1}^{n-1} \frac{(n-1)^{n-k-1}}{(n-k-1)!} = (n-2)! e^{n-1} \IP(0\leq \Poi(n-1)\leq n-2)\\
  &\sim \frac{1}{2} \sqrt{2\pi} (n-1)^{n-1+1/2 -1},
 \end{align*}
 which simplifies to the desired asymptotics for $c'(n)$. For the second part, we again fix $x>0$, set $R=\floor{x\sqrt{n}}$ and compute, as before,
 \begin{align*}
  \IP(L'_n \leq R)
  &= \frac{1}{c'(n)} \sum_{k=1}^R a'(n,k)
   = \frac{(n-2)!}{c'(n)} e^{n-1} \IP(n-1-R\leq \Poi(n-1)\leq n-2) \\
  &\sim 2\int_{-\frac{R}{\sqrt{n-1}}}^{-\frac{1}{\sqrt{n-1}}} \frac{1}{\sqrt{2\pi}}  e^{-y^2/2} \Di y
   \sim \int_{-x}^x \frac{1}{\sqrt{2\pi}}  e^{-y^2/2} \Di y. \qedhere
 \end{align*}
\end{proof}

\subsection{Connected regular multiset permutations}

Fix $r\geq 2$ and consider permutations of the \emph{$r$-regular multiset} $M=(1^{r},2^{r},\dots,j^{r})$. Imitating the calculations in Example \ref{ex}, we find that
\[
 t(M,k) = \binom{j}{k} k! \cdot k\cdot \frac{(jr-k-1)!}{r!^{j-k}(r-1)!^{k-1}(r-2)!},
\]
which reduces to the special case in Example \ref{ex} when $j=k$.
Indeed, we first choose $k$ distinct elements from $[j]$ to be our terminal closers. They can appear in any order among the first $k$ elements of the permutation. The $(k+1)$-th element needs to repeat one of these $k$ values. The remaining $jr-k-1$ positions are filled with the remaining elements from the multiset. By Corollary \ref{cor:t(M,k)-and-a(M,k)}, we thus have
\begin{equation*}
a(M,k)
 = \binom{j}{k} k! \frac{(jr-k-1)!}{r!^{j-k}(r-1)!^{k-1}(r-2)!}
 = \frac{r-1}{r!^{j}} \binom{j}{k} k! r^k (jr-k-1)!.
\end{equation*}
Let $c_M(j) := \sum_{k=1}^j a(M,k)$ be the total number of connected multiset permutations.

\begin{theorem}\label{thm:regM}
 For $r\geq 2$ fixed and $M$ as above, we have
 \begin{equation}\label{eq:thmregM1}
  c_M(j)
  \sim j^{1/2} \frac{(jr-1)!}{r!^{j}} \sqrt{\frac{\pi r(r-1)}{2}}
 \end{equation}
 for $j\to\infty$. Moreover, let $L_j$ denote the number of cyclic points in a uniform random connected multiset permutation of $M$. Then
 \begin{equation*}
  j^{-1/2} L_j \dto \left|N\left(0,\frac{r}{r-1}\right)\right|.
 \end{equation*}
\end{theorem}

\begin{proof}
We argue as in Section \ref{subsec:cpf}, just with a negative binomial distribution instead of Poisson. We have
\begin{align*}
    c_M(j) &= \sum_{k=1}^j a(M,k)
    = \sum_{k=1}^j \frac{r-1}{r!^{j}} \binom{j}{k} k! r^k (jr-k-1)! \\
    &= \frac{r-1}{r!^{j}} j! \sum_{k=1}^j  \frac{r^k (jr-k-1)!}{(j-k)!}
    = \frac{r-1}{r!^{j}} j! \sum_{h=0}^{j-1}  \frac{r^{j-h} (jr-j+h-1)!}{h!} \\
    &= \frac{(r-1)r^j j!(jr-j-1)!}{r!^{j}}  \sum_{h=0}^{j-1} \binom{jr-j+h-1}{h} r^{-h} \\
    &= \frac{r^{jr} j!(jr-j-1)!}{(r-1)^{jr-j-1}r!^{j}}  \sum_{h=0}^{j-1} \binom{jr-j+h-1}{h} (1-r^{-1})^{jr-j} r^{-h} \\
    &= \frac{r^{jr} j!(jr-j-1)!}{(r-1)^{jr-j-1}r!^{j}} 
    \IP\big(\NB(jr-j,1-r^{-1})\leq j-1\big).
\end{align*}
The probability $\IP(\cdot)$ tends to $1/2$ as in the previous instances. Applying Stirling's formula, one gets
\begin{align*}
 \frac{j!(jr-j-1)!}{(jr-1)!} 
 &\sim \sqrt{\frac{2\pi j(jr-j-1)}{jr-1}} \cdot \frac{j^{j+jr-j-1}}{j^{jr-1}} \frac{\left(r-1-\frac{1}{j}\right)^{jr-j-1}}{\left(r-\frac{1}{j}\right)^{jr-1}}\\
 &\sim j^{1/2} \sqrt{\frac{2\pi (r-1)}{r}} \frac{(r-1)^{jr-j-1} e^{-1}}{r^{jr-1} e^{-1}}.
\end{align*}
Plugging this into the full expression for $c_M(j)$ leads to some more cancellations and \eqref{eq:thmregM1}.

For the second part, we again fix $x>0$, set $R=\floor{x\sqrt{j}}$ and compute
 \begin{align*}
  \IP( L_j \leq R)
  &= \frac{1}{c_M(j)} \sum_{k=1}^R a(M,k)
   \sim
    2\IP\big(j-R\leq \NB(jr-j,1-r^{-1})\leq j-1\big), \\
  &\sim 2\int_{-\frac{R}{\sqrt{j\frac{r}{r-1}}}}^{-\frac{1}{\sqrt{j\frac{r}{r-1}}}} \frac{1}{\sqrt{2\pi}} e^{-y^2/2} \Di y
   \sim \int_{-x}^x \frac{1}{\sqrt{2\pi \frac{r}{r-1}}} e^{-\frac{y^2}{2\frac{r}{r-1}}} \Di y,
 \end{align*}
 which is the cumulative distribution function of $\left|N\left(0,\frac{r}{r-1}\right)\right|$.
\end{proof}

\section*{Acknowledgements}

The authors thank J. E. Paguyo, Martin Rubey, and Richard Stanley for helpful discussions.

\end{document}